\documentclass[12pt]{amsart}

\usepackage{amsmath,amssymb,amsfonts,amsthm,latexsym,graphicx,multirow}
\usepackage{hyperref}

\def\A{\mathrm{A}}   \def\Aut{\mathrm{Aut}}
 
\def\C{\mathrm{C}}  \def\Cay{\mathrm{Cay}} \def\Cen{\mathbf{C}}  
\def\D{\mathrm{D}}

   \def\GL{\mathrm{GL}}  \def\GU{\mathrm{GU}}

 \def\magma{{\sc Magma} }

\def\Nor{\mathbf{N}}
\def\Out{\mathrm{Out}}
  \def\PGL{\mathrm{PGL}}     \def\PSL{\mathrm{PSL}}     \def\PSU{\mathrm{PSU}}

\def\SL{\mathrm{SL}}         

\def\Z{\mathbf{Z}} 

\newtheorem{theorem}{Theorem}[section]
\newtheorem{lemma}[theorem]{Lemma}
\newtheorem{proposition}[theorem]{Proposition}

\theoremstyle{definition}

\newtheorem{conjecture}[theorem]{Conjecture}
\newtheorem{problem}[theorem]{Problem}

\begin{document}

\title{Cubic graphical regular representations of $\PSL_2(q)$}

\author[Xia]{Binzhou Xia}
\address{Beijing International Center for Mathematical Research\\
Peking University\\
Beijing, 100871\\
P. R. China}
\email{binzhouxia@pku.edu.cn}

\author[Fang]{Teng Fang}
\address{Beijing International Center for Mathematical Research\\
Peking University\\
Beijing, 100871\\
P. R. China}
\email{tengfang@pku.edu.cn}

\maketitle

\begin{abstract}
We study cubic graphical regular representations of the finite simple groups $\PSL_2(q)$. It is shown that such graphical regular representations exist if and only if $q\neq7$, and the generating set must consist of three involutions.
\end{abstract}


\section{Introduction}

Given a group $G$ and a subset $S\subset G$ such that $1\notin S$ and $S=S^{-1}:=\{g^{-1}\mid g\in S\}$,
the \emph{Cayley graph} $\Cay(G,S)$ of $G$ is the graph with vertex set $G$ such that two vertices $x,y$ are adjacent
if and only if $yx^{-1}\in S$. It is easy to see that $\Cay(G,S)$ is connected if and only if $S$ generates the group $G$. If one identifies $G$ with its right regular representation, then $G$ is a subgroup of $\Aut(\Cay(G,S))$. We call $\Cay(G,S)$ a \emph{graphical regular representation} (\emph{GRR} for short) of $G$ if $\Aut(\Cay(G,S))=G$. The problem of seeking graphical regular representations for given groups has been investigated for a long time. A major accomplishment for this problem is the determination of finite groups without a GRR, see~\cite[16g]{Biggs1993}. It turns out that most finite groups admit at least one GRR. For instance, every finite unsolvable group has a GRR~\cite{Godsil1981}.

In contrast to unrestricted GRRs, the question of whether a group has a GRR of prescribed valency is largely open. Research on this subject have been focusing on small valencies~\cite{FLWX2002,Godsil1983,XX2004}. In~2002, Fang, Li, Wang and Xu~\cite{FLWX2002} issued the following conjecture.

\begin{conjecture}\label{conj1}
(\cite[Remarks on Theorem~1.3]{FLWX2002}) Every finite nonabelian simple group has a cubic GRR.
\end{conjecture}

Note that any GRR of a finite simple group must be connected, for otherwise its full automorphism group would be a wreath product. Hence if $\Cay(G,S)$ is a GRR of a finite simple group $G$, then $S$ is necessarily a generating set of $G$. Apart from a few small groups, Conjecture~\ref{conj1} was only known to be true for the alternating groups~\cite{Godsil1983} and Suzuki groups~\cite{FLWX2002}, while no counterexample was found yet. In this paper, we study cubic GRRs for finite projective special linear groups of dimension two. In particular, Theorem~\ref{thm4} shows that Conjecture~\ref{conj1} fails for $\PSL_2(7)$ but holds for all $\PSL_2(q)$ with $q\neq7$.

For any subset $S$ of a group $G$, denote by $\Aut(G,S)$ the group of automorphisms of $G$ fixing $S$ setwise. Each element in $\Aut(G,S)$ is an automorphism of $\Cay(G,S)$ fixing the identity of $G$. Hence a necessary condition for $\Cay(G,S)$ to be a GRR of $G$ is that $\Aut(G,S)=1$. In~\cite{FLWX2002}, the authors showed that this condition is also sufficient for many cubic Cayley graphs of finite simple groups. We state their result for simple groups $\PSL_2(q)$ as follows, which is the starting point of the present paper.

\begin{theorem}\label{thm3}
\emph{(\cite{FLWX2002})} Let $G=\PSL_2(q)$ be a simple group, where $q\neq11$ is a prime power, and $S$ be a generating set of $G$ with $S^{-1}=S$ and $|S|=3$. Then $\Cay(G,S)$ is a GRR of $G$ if and only if $\Aut(G,S)=1$.
\end{theorem}

The following are our three main results.

\begin{theorem}\label{thm4}
For any prime power $q\geqslant5$, $\PSL_2(q)$ has a cubic GRR if and only if $q\neq7$.
\end{theorem}

\begin{theorem}\label{thm2}
For each prime power $q$ there exist involutions $x$ and $y$ in $\PSL_2(q)$ such that the probability for a randomly chosen involution $z$ to make
$$
\Cay(\PSL_2(q),\{x,y,z\})
$$
a cubic GRR of $\PSL_2(q)$ tends to $1$ as $q$ tends to infinity.
\end{theorem}

\begin{proposition}\label{thm1}
Let $q\geqslant5$ be a prime power and $G=\PSL_2(q)$. If $\Cay(G,S)$ is a cubic GRR of $G$, then $S$ is a set of three involutions.
\end{proposition}

Theorem~\ref{thm2} shows that it is easy to make GRRs for $\PSL_2(q)$ from three involutions. On the other hand, Proposition~\ref{thm1} says that one can only make GRRs for $\PSL_2(q)$ from three involutions, which is a response to~\cite[Problem~1.2]{Godsil1983} as well. (Note that for a cubic Cayley graph $\Cay(G,S)$, the set $S$ either consists of three involutions, or has the form $\{x,y,y^{-1}\}$ with $o(x)=2$ and $o(y)>2$.) The proof of Theorem~\ref{thm2} is at the end of Section~\ref{sec1}, and the proofs of Theorem~\ref{thm4} and Proposition~\ref{thm1} are in Section~\ref{sec2}. We also pose two problems concerning cubic GRRs for other families of finite simple groups at the end of this paper.

\section{Preliminaries}

The following result is well known, see for example~\cite[II~\S7 and~\S8]{Huppert1967}.

\begin{lemma}\label{lem3}
Let $q\geqslant5$ be a prime power and $d=\gcd(2,q-1)$. Then $\PGL_2(q)$ has a maximal subgroup $M=\D_{2(q+1)}$. Moreover, $M\cap\PSL_2(q)=\D_{2(q+1)/d}$, and for $q\notin\{7,9\}$ it is maximal in $\PSL_2(q)$.
\end{lemma}

The next lemma concerns facts about involutions in two-dimensional linear groups which is needed in the sequel.

\begin{lemma}\label{lem1}
Let $q=p^f\geqslant5$ for some prime $p$ and $G=\PSL_2(q)$. Then the following hold.
\begin{itemize}
\item[(a)] There is only one conjugacy class of involutions in $G$.
\item[(b)] For any involution $g$ in $G$,
$$
\Cen_G(g)=
\begin{cases}
\C_2^f,\quad\textup{if $p=2$},\\
\D_{q-1},\quad\textup{if $q\equiv1\pmod{4}$},\\
\D_{q+1},\quad\textup{if $q\equiv3\pmod{4}$}.
\end{cases}
$$
\item[(c)] If $p>2$, then for any involution $\alpha$ in $\PGL_2(q)$, the number of involutions in $\Cen_G(\alpha)$ is at most $(q+3)/2$.
\end{itemize}
\end{lemma}

\begin{proof}
Parts~(a) and~(b) can be found in~\cite[Lemma~A.3]{GZ2010}. To prove part~(c), assume that $p>2$ and $\alpha$ is an involution in $\PGL_2(q)$. By \cite[Lemma~A.3]{GZ2010} we have $\Cen_G(\alpha)=\D_{q+\varepsilon}$ with $\varepsilon=\pm1$. As a consequence, the number of involutions in $\Cen_G(\alpha)$ is at most $1+(q+\varepsilon)/2\leqslant(q+3)/2$. This completes the proof.
\end{proof}

\section{GRRs from three involutions}\label{sec1}

Recall from Lemma~\ref{lem3} that $\PSL_2(q)$ has a maximal subgroup $\D_{2(q+1)/d}$, where $d=\gcd(2,q-1)$. The following proposition plays the central role in this paper.

\begin{proposition}\label{prop1}
Let $q=p^f\geqslant11$ for some prime $p$, $d=\gcd(2,q-1)$, $G=\PSL_2(q)$, and $H=\D_{2(q+1)/d}$ be a maximal subgroup of $G$. Then for any two involutions $x,y$ with $\langle x,y\rangle=H$, there are at least
$$
\frac{q^2-4d^2fq-(d+2)q-4d^2f-3d^2+2d-1}{d}
$$
involutions $z\in G$ such that $\langle x,y,z\rangle=G$ and $\Aut(G,\{x,y,z\})=1$.
\end{proposition}

\begin{proof}
Fix involutions $x,y$ in $H$ such that $\langle x,y\rangle=H$. Identify the elements in $G$ with their induced inner automorphisms of $G$. In this way, $G$ is a normal subgroup of $A:=\Aut(G)$, and the elements of $A$ act on $G$ by conjugation. Denote by $V$ the set of involutions in $G$, and
\begin{eqnarray*}
L&=&\{y^\alpha\mid\alpha\in A,x^\alpha=x\}\cup\{y^\alpha\mid\alpha\in A,x^\alpha=y\}\\
&&\cup\ \{x^\alpha\mid\alpha\in A,y^\alpha=x\}\cup\{x^\alpha\mid\alpha\in A,y^\alpha=y\}.
\end{eqnarray*}
By Lemma~\ref{lem1}, $x$ and $y$ are conjugate in $A$. Hence
\begin{eqnarray}\label{eq1}
|L|&\leqslant&|\{\alpha\in A\mid x^\alpha=x\}|+|\{\alpha\in A\mid x^\alpha=y\}|\\
\nonumber&&+\ |\{\alpha\in A\mid y^\alpha=x\}|+|\{\alpha\in A\mid y^\alpha=y\}|\\
\nonumber&=&4|\Cen_A(x)|\leqslant4|\Out(G)||\Cen_G(x)|=4df|\Cen_G(x)|\leqslant4df(q+1)
\end{eqnarray}
by virtue of Lemma~\ref{lem1}. Denote by $I$ the set of involutions $\alpha\in A$ such that $x^\alpha=y$ and $y^\alpha=x$.

Take an arbitrary $\alpha\in I$. Then $H^\alpha=\langle x^\alpha,y^\alpha\rangle=\langle y,x\rangle=H$, i.e., $\alpha\in\Nor_A(H)$. Write $\Nor_A(H)=\langle a\rangle\rtimes\langle b\rangle$ with $\langle a\rangle=\C_{q+1}<\PGL_2(q)$ and $\langle b\rangle=\C_{2f}$, so that $H=\Nor_A(H)\cap G=\langle a^d\rangle\rtimes\langle b^f\rangle$. Since $x$ and $y$ are two involutions generating $H$, we have $x=a^{di}b^f$ and $y=a^{dj}b^f$ for some integers $i$ and $j$. Moreover, either $\alpha=a^kb^f$ for some integer $k$, or $\alpha=a^{(q+1)/2}$ with $q$ odd, because $\alpha$ is an involution in $\Nor_A(H)$. However, if $\alpha=a^{(q+1)/2}$ with $q$ odd, then $\alpha$ will fix $x$ and $y$, respectively. Thus $\alpha=a^kb^f$ for some integer $k$, and in particular, $\alpha\in\PGL_2(q)$. In view of $a^{b^f}=a^{-1}$, one computes that
$$
x^\alpha=(a^{di}b^f)^{a^kb^f}=(a^{di})^{a^kb^f}(b^f)^{b^fa^{-k}}=(a^{di})^{b^f}(b^f)^{a^{-k}}=a^{2k-di}b^f.
$$
This together with the assumption $x^\alpha=y=a^{dj}b^f$ gives $(a^k)^2=a^{d(i+j)}$. As a consequence, $|I|\leqslant d$. If $p=2$, then $\alpha\in\PGL_2(q)=G$, and we see from Lemma~\ref{lem1} that $|V\cap\Cen_G(\alpha)|=q-1$. If $p$ is odd, then Lemma~\ref{lem1} asserts that $|V\cap\Cen_G(\alpha)|\leqslant(q+3)/2$. To sum up, we have
\begin{equation*}
|\bigcup\limits_{\alpha\in I}(V\cap\Cen_G(\alpha))|\leqslant|I|\cdot\frac{q+3d-3}{d}\leqslant d\cdot\frac{q+3d-3}{d}=q+3d-3.
\end{equation*}
Due to Lemma~\ref{lem1} and the orbit-stabilizer theorem, $|V|\geqslant|G|/(q+1)=q(q-1)/d$, whence
\begin{eqnarray}\label{eq2}
&&|(V\setminus H)\setminus\bigcup\limits_{\alpha\in I}\Cen_G(\alpha)|\geqslant|V|-|V\cap H|-|\bigcup\limits_{\alpha\in I}(V\cap\Cen_G(\alpha))|\\
\nonumber&\geqslant&\frac{q(q-1)}{d}-\left(\frac{q+1}{d}+1\right)-(q+3d-3)=\frac{q^2-(d+2)q-3d^2+2d-1}{d}.
\end{eqnarray}

Suppose that $z$ is an involution of $G$ outside $H$. It follows that $\langle x,y,z\rangle=G$. If $\alpha$ is a nonidentity in $\Aut(G,\{x,y,z\})$, then either $z^\alpha=z$ or $z\in L$. In the former case, $\alpha$ is an involution as $\alpha$ interchanges $x$ and $y$, and hence $z\in\bigcup_{\alpha\in I}\Cen_G(\alpha)$. This implies that for any $z$ in $(V\setminus H)\setminus\bigcup_{\alpha\in I}\Cen_G(\alpha)$ but not in $L$, one has $\langle x,y,z\rangle=G$ and $\Aut(G,\{x,y,z\})=1$. Now combining~\eqref{eq1} and~\eqref{eq2} we deduce that the number of choices of such $z$ is at least
\begin{eqnarray*}
|(V\setminus H)\setminus\bigcup\limits_{\alpha\in I}\Cen_G(\alpha)|-|L|&\geqslant&\frac{q^2-(d+2)q-3d^2+2d-1}{d}-4df(q+1)\\
&=&\frac{q^2-4d^2fq-(d+2)q-4d^2f-3d^2+2d-1}{d},
\end{eqnarray*}
as the lemma asserts.
\end{proof}

As an immediate consequence of Proposition~\ref{prop1}, we give a proof for Theorem~\ref{thm2}.

\vskip0.1in
\noindent\textbf{Proof of Theorem~\ref{thm2}:} Assume without loss of generality that $q\geqslant11$, and take $x,y$ to be any two involutions generating the maximal subgroup $\D_{2(q+1)/d}$ of $\PSL_2(q)$, where $d:=\gcd(2,q-1)$. Let $J=J_{q,x,y}$ be the set of involutions $z$ such that $\Cay(\PSL_2(q),\{x,y,z\})$ is a GRR of $\PSL_2(q)$. By Theorem~\ref{thm3}, $J$ equals the set of involutions $z$ such that $\langle x,y,z\rangle=\PSL_2(q)$ and $\Aut(\PSL_2(q),\{x,y,z\})=1$. Hence applying Proposition~\ref{prop1} we obtain
\begin{eqnarray*}
|J|&\geqslant&\frac{q^2-4d^2fq-(d+2)q-4d^2f-3d^2+2d-1}{d}\\
\nonumber&\geqslant&\frac{q^2-16fq-4q-16f-9}{d}\geqslant\frac{q^2-16q\log_2q-4q-16\log_2q-9}{d}.
\end{eqnarray*}
Moreover, combining Lemma~\ref{lem1} and the orbit-stabilizer theorem shows that the number of involutions in $\PSL_2(q)$ is at most $|\PSL_2(q)|/(q-1)=q(q+1)/d$. Thus for a randomly chosen involution $z$, the probability such that $\Cay(\PSL_2(q),\{x,y,z\})$ is
a cubic GRR of $\PSL_2(q)$ is at least
$$
\frac{|J|}{q(q+1)/d}\geqslant\frac{q^2-16q\log_2q-4q-16\log_2q-9}{q(q+1)},
$$
which tends to $1$ as $q$ tends to infinity.
\qed

\section{Conclusion}\label{sec2}

\begin{lemma}\label{lem5}
Let $q=p^f$ for some prime $p$, and $d=\gcd(2,q-1)$. If $q\geqslant29$ or $q=23$, then
$$
q^2-4d^2fq-(d+2)q-4d^2f-3d^2+2d-1>0.
$$
\end{lemma}

\begin{proof}
It is direct to verify the conclusion for $q\in\{32,64,81,128,256\}$. Hence we only need to prove the lemma with $q\notin\{32,64,81,128,256\}$. Note that under this assumption, $q>22f$. Since $q>25/2\geqslant(16f+9)/(6f-4)$, or equivalently $4q+16f+9<6fq$, we have
\begin{eqnarray*}
&&q^2-4d^2fq-(d+2)q-4d^2f-3d^2+2d-1\\
&\geqslant& q^2-16fq-4q-16f-9>q^2-16fq-6fq=q(q-22f)>0
\end{eqnarray*}
as desired.
\end{proof}

\begin{lemma}\label{lem6}
Let $q$ be an odd prime power, and $x$ be an involution of $\PSL_2(q)$. If $q\equiv3\pmod{4}$, then for any $y\in\PSL_2(q)$, there exists $\alpha\in\Aut(\PSL_2(q))$ such that $x^\alpha=x$ and $y^\alpha=y^{-1}$.
\end{lemma}

\begin{proof}
Appeal to the isomorphism $\PSL_2(q)\cong\PSU_2(q)$. Let $i$ be an element of order four in $\mathbb{F}_{q^2}^\times$, $G=\PSU_2(q)$, $A=\Aut(G)$ and $\psi$ be the homomorphism from $\GU_2(q)$ to $G$ modulo $\Z(\GU_2(q))$. In view of Lemma~\ref{lem1} we may assume that
$$
x=
\begin{pmatrix}
i&0\\
0&-i
\end{pmatrix}
^\psi
$$
($x$ is indeed an element of $\PSL_2(q)$ since $i^2=-1$). By~\cite[II~8.8]{Huppert1967} one has
$$
y=
\begin{pmatrix}
a&b\\
-b^q&a^q
\end{pmatrix}
^\psi,
$$
where $a,b\in\mathbb{F}_{q^2}$ such that $a^{q+1}+b^{q+1}=1$. Then set $c=b^{q-1}$ if $b\neq0$, and $c=1$ if $b=0$. Define
$$
\alpha:g\mapsto
\begin{pmatrix}
0&c^q\\
1&0
\end{pmatrix}
^\psi g
\begin{pmatrix}
0&1\\
c&0
\end{pmatrix}
^\psi
$$
for any $g\in G$. It is straightforward to verify that $\alpha\in A$, $x^\alpha=x$ and $y^\alpha=y^{-1}$. Thus the lemma is true.
\end{proof}

We remark that if $q\not\equiv3\pmod{4}$ then the conclusion of Lemma~\ref{lem6} may not hold. For example, when $q=8$ or $13$, respectively, there exist $x,y\in\PSL_2(q)$ with $o(x)=2$ and $o(y)>2$ such that $\{x,y\}$ does not generate $\PSL_2(q)$ and $\Aut(\PSL_2(q),\{x,y,y^{-1}\})=1$.

\vskip0.1in
Now we are able to prove Theorem~\ref{thm4} and Proposition~\ref{thm1}.

\vskip0.1in
\noindent\textbf{Proof of Theorem~\ref{thm4}:} For $q\in\{8,9,11,13,16,17,19,25,27\}$, computation in \magma~\cite{magma} shows that there exist involutions $x,y,z\in\PSL_2(q)$ such that $\langle x,y,z\rangle=\PSL_2(q)$ and $\Aut(\PSL_2(q),\{x,y,z\})=1$ (one can further take $x,y$ to be the generators of $\D_{2(q+1)/\gcd(2,q-1)}$). This together with Proposition~\ref{prop1} and Lemma~\ref{lem5} indicates that whenever $q\geqslant8$, there exist involutions $x,y,z\in\PSL_2(q)$ such that $\langle x,y,z\rangle=\PSL_2(q)$ and $\Aut(\PSL_2(q),\{x,y,z\})=1$. Hence by Theorem~\ref{thm3}, $\PSL_2(q)$ has a cubic GRR if $q\geqslant8$ and $q\neq11$. Moreover, the existence of cubic GRRs for $\PSL_2(5)$ and $\PSL_2(11)$ was proved in~\cite[Remarks on Theorem~1.3]{FLWX2002}. Therefore, $\PSL_2(q)$ has a cubic GRR provided $q\neq7$.

It remains to show that $\PSL_2(7)$ has no cubic GRR. Suppose on the contrary that $\Cay(\PSL_2(7),S)$ is a cubic GRR of $\PSL_2(7)$. Then according to Lemma~\ref{lem6}, $S$ must be a set of three involutions. However, computation in \magma~\cite{magma} shows that for any involutions $x,y,z$ with $\langle x,y,z\rangle=\PSL_2(7)$, there exists an involution $\alpha\in\PGL_2(7)$ such that $\{x^\alpha,y^\alpha,z^\alpha\}=\{x,y,z\}$. This contradiction completes the proof.
\qed

\vskip0.1in
\noindent\textbf{Proof of Proposition~\ref{thm1}:} By contradiction, suppose that $S$ does not consist of three involutions. Then since $1\notin S$ and $S=S^{-1}$, we deduce that $S=\{x,y,y^{-1}\}$ with $o(x)=2$ and $o(y)>2$. It follows that $\langle x,y\rangle=G$.

First assume that $q$ is even. By virtue of Lemma~\ref{lem1}, $x$ can be taken as any involution in $G=\SL_2(q)$, whence we may assume that
$$
x=
\begin{pmatrix}
1&1\\
0&1
\end{pmatrix}.
$$
Suppose
$$
y=
\begin{pmatrix}
a&b\\
c&d
\end{pmatrix},
$$
where $a,b,c,d\in\mathbb{F}_q$ such that $ad-bc=1$. If $c=0$, then $\langle x,y\rangle$ is contained in the group of upper triangular matrices in $\SL_2(q)$, impossible because $\langle x,y\rangle=G$. Consequently, $c\neq0$. Set
$$
h=
\begin{pmatrix}
c&a+d\\
0&c
\end{pmatrix},
$$
and define $\alpha:g\mapsto h^{-1}gh$ for any $g\in G$. Then $\alpha\in A$, and one sees easily that $x^\alpha=x$ and $y^\alpha=y^{-1}$. This implies $\Aut(G,S)\neq1$, contrary to the condition that $\Cay(G,S)$ is a GRR of $G$.

Next assume that $q\equiv1\pmod{4}$. Let $\omega$ be an element of order four in $\mathbb{F}_q^\times$, and $\varphi$ be the homomorphism from $\GL_2(q)$ to $\PGL_2(q)$ modulo $\Z(\GL_2(q))$. Due to Lemma~\ref{lem1}, we may assume that
$$
x=
\begin{pmatrix}
\omega&0\\
0&-\omega
\end{pmatrix}
^\varphi
$$
($x$ is indeed an element of $\PSL_2(q)$ since $\omega^2=-1$). Suppose
$$
y=
\begin{pmatrix}
a&b\\
c&d
\end{pmatrix}
^\varphi,
$$
where $a,b,c,d\in\mathbb{F}_q$ such that $ad-bc=1$. If $bc=0$, then the preimage of $\langle x,y\rangle$ under $\varphi$ is contained either in the group of upper triangular matrices or in the group of lower triangular matrices in $\SL_2(q)$. Hence $bc\neq0$ as $\langle x,y\rangle=G$. Set
$$
h=
\begin{pmatrix}
0&b\\
-c&0
\end{pmatrix}
^\varphi,
$$
and define $\alpha:g\mapsto h^{-1}gh$ for any $g\in G$. Then $\alpha\in A$, and it is direct to verify that $x^\alpha=x$ and $y^\alpha=y^{-1}$. This implies $\Aut(G,S)\neq1$, contrary to the condition that $\Cay(G,S)$ is a GRR of $G$.

Finally assume that $q\equiv3\pmod{4}$. Then we derive from Lemma~\ref{lem6} that $\Aut(G,S)\neq1$, again a contradiction to the condition that $\Cay(G,S)$ is a GRR of $G$. The proof is thus completed.
\qed

We conclude with two problems on cubic GRRs for other families of finite simple groups. First, in view of Theorem~\ref{thm4}, a natural problem is to determine which finite nonabelian simple groups have no cubic GRR. Such groups would be vary rare, and we conjecture that there are only finitely many of them.

\begin{conjecture}
There are only finitely many finite nonabelian simple groups that have no cubic GRR.
\end{conjecture}

\begin{problem}
Classify the finite nonabelian simple groups that have no cubic GRR.
\end{problem}

Second, as it is shown in Proposition~\ref{thm1} that for finite simple groups $\PSL_2(q)$, a GRR can only be made from three involutions, one would ask what is the situation for other finite nonabelian simple groups, which is the problem below.

\begin{problem}
Classify the finite nonabelian simple groups $G$ having no GRR of form $\Cay(G,\{x,y,y^{-1}\})$, where $o(x)=2$ and $o(y)>2$.
\end{problem}

\noindent\textsc{Acknowledgements.}
The first author acknowledges the support of China Postdoctoral Science Foundation Grant 2014M560838 and National Science Foundation of China grant 11501011. The authors would like to thank the anonymous referees for their comments to improve the presentation.

\end{document}